\documentclass[12pt,a4,notitlepage]{amsart}
\usepackage{amsfonts}

\usepackage{graphicx}
\usepackage{amscd}

\newtheorem{theorem}{Theorem}
\theoremstyle{plain}

\newtheorem{corollary}[theorem]{Corollary}

\newtheorem{example}[theorem]{Example}

\newtheorem{lemma}[theorem]{Lemma}

\newtheorem{problem}[theorem]{Problem}
\newtheorem{proposition}[theorem]{Proposition}
\newtheorem{remark}[theorem]{Remark}

\numberwithin{equation}{section}

\begin{document}
\title[On the separable quotient problem]{On the separable quotient problem for Banach spaces}
\author{J. C. Ferrando, J. K{\c{a}}kol, M. L\'{o}pez-Pellicer and W. \'Sliwa}
\address{Centro de Investigaci\'{o}n Operativa, Edificio Torretamarit, Avda
de la Universidad,%
\newline
\indent%
Universidad Miguel Hern\'{a}ndez, E-03202 Elche (Alicante). SPAIN}
\email{jc.ferrando@umh.es}
\address{Faculty of Mathematics and Informatics. A. Mickiewicz University,
61-614 Pozna{\'{n}},%
\newline
\indent%
Poland}
\email{kakol@amu.edu.pl}
\address{Depto. de Matem\'{a}tica Aplicada and IMPA. Universitat Polit\`{e}%
cnica de Val\`{e}ncia,%
\newline
\indent%
E-46022 Valencia, Spain }
\email{mlopezpe@mat.upv.es}
\address{Faculty of Mathematics and Natural Sciences University of Rzesz\'{o}%
w,%
\newline
\indent%
35-310 Rzesz\'{o}w, Poland }
\email{sliwa@amu.edu.pl}
\thanks{The first three named authors were supported by Grant
PROMETEO/2013/058 of the Conselleria d'Educaci\'{o}, Investigaci\'{o},
Cultura i Esport of Generalitat Valenciana. The second named author was also
supported by GA\v{C}R Project 16-34860L and RVO: 67985840.}
\dedicatory{To the memory of our Friend Professor Pawe{\l} Doma\'nski}

\begin{abstract}
While the classic \emph{separable quotient problem } remains open, we survey
general results related to this problem and examine the existence of a
particular infinite-dimensional separable quotient in some Banach spaces of
vector-valued functions, linear operators and vector measures. Most of the
results presented are consequence of known facts, some of them relative to
the presence of complemented copies of the classic sequence spaces $c_{0}$
and $\ell _{p}$, for $1\leq p\leq \infty $. Also  recent results  of Argyros, Dodos, Kanellopoulos \cite{A}  and   \'Sliwa \cite{sliwa} are provided.  This makes our presentation
supplementary to a previous survey (1997) due to Mujica.
\end{abstract}

\maketitle

\section{Introduction}

One of unsolved problems of Functional Analysis (posed by S. Mazur in 1932)
asks:

\begin{problem}
\label{mazur} Does any infinite-dimensional Banach space have a separable
\emph{(}infinite dimensional\emph{)} quotient?
\end{problem}

An easy application of the open mapping theorem shows that an infinite
dimensional Banach space $X$ has a separable quotient if and only if $X$ is
mapped on a separable Banach space under a continuous linear map.

Seems that the first comments about Problem \ref{mazur} are mentioned in %
\cite{La} and \cite{Ro}. It is already well known that all reflexive, or
even all infinite-dimensional weakly compactly generated Banach spaces (WCG
for short), have separable quotients. In \cite[Theorem IV.1(i)]{JR} Johnson
and Rosenthal proved that every infinite dimensional separable Banach space
admits a quotient with a Schauder basis. The latter result provides another
(equivalent) reformulation of Problem \ref{mazur}.

\begin{problem}
\label{mazur2} Does any infinite dimensional Banach space admits an infinite
dimensional quotient with a Schauder basis?
\end{problem}

In \cite[Theorem 2]{johnson} is proved that if $Y$ is a separable closed
subspace of an infinite-dimensional Banach space $X$ and the quotient $X/Y$
has a separable quotient, then $Y$ is quasi-complemented in $X$. So one gets
the next equivalent condition to Problem \ref{mazur}.

\begin{problem}
\label{mazur3} Does any infinite dimensional Banach space $X$ contains a
separable closed subspace which has a proper quasi-complement in $X$?
\end{problem}

Although Problem \ref{mazur} is left open for Banach spaces, a corresponding
question whether any infinite dimensional metrizable and complete non-normed
topological vector space $X$ admits a separable quotient has been already
solved. Indeed, if $X$ is locally convex (i.\thinspace e., $X$ is a Fr\'{e}%
chet space), a result of Eidelheit \cite{eidel} ensures that $X$ has a
quotient isomorphic to $\mathbb{K}^{\mathbb{N}}$, where $\mathbb{K}\in \{%
\mathbb{R},\mathbb{C}\}$.

In \cite{drew} Drewnowski posed a more general question: whether any
infinite dimensional metrizable and complete topological vector space $X$
has a closed subspace $Y$ such that the dimension of the quotient $X/Y$ is\
the continuum (in short $\dim (X/Y)=\mathfrak{c})$. The same paper contains
an observation stating that any infinite dimensional Fr\'{e}chet locally
convex space admits a quotient of dimension $\mathfrak{c}:=2^{\aleph _{0}}$.
Finally Drewnowski's problem has been solved by M. Popov, see \cite{popov}.
He showed, among others, that for $0<p<1$ the space $X:=L_{p}([0,1]^{2^{%
\mathfrak{c}}})$ does not admit a proper closed subspace $Y$ such that $\dim
(X/Y)\leq \mathfrak{c}$. Consequently $X$ does not have a separable quotient.

The organization of the present paper goes as follows. In the second section
we gather general selected results about the separable quotient problem and
some classic results on Banach spaces containing copies of sequence spaces,
providing as well some straightforward consequences. Recent   results of Argyros, Dodos, Kanellopoulos \cite{A}  and   \'Sliwa \cite{sliwa} are also provided.

In the third section,
among others, we exhibit how the weak*-compactness of the dual unit ball is
related to the existence of quotients isomorphic to $c_{0}$ or $\ell _{2}$.
Next section contains three classic results about complete tensor products
of Banach spaces and their applications to the separable quotient problem.

Last sections are devoted to examine the existence of separable quotients
for many concrete classes of `big' Banach spaces, as Banach spaces of
vector-valued functions, bounded linear operators and vector measures. For
these classes of spaces we fix particular separable quotients by means of
the literature on the subject and the previous results. This line of
research has been also continued in a more general setting for the class of
topological vector spaces, particularly for spaces $C(X)$ of real-valued
continuous functions endowed with the pointwise and compact-open topology,
see \cite{ks}, \cite{sa1} and \cite{sa2}.

This paper intends to supplement Mujica's survey article \cite{Mu} by
collecting together some results not mentioned there, adding new facts
published afterwards and gathering some results being consequence of known
facts, some of them relative to the presence of complemented copies of the
classic spaces $c_{0}$ and $\ell _{p}$, for $1\leq p\leq \infty $, hoping
this will be useful for\ researchers interested in this area.

\section{A few results for general Banach spaces}

Let us start with the following remarkable concrete result of Agryros, Dodos and Kanellopoulos \cite{A} related to
Problem \ref{mazur}.  They proved (using the Ramsey Theory)  that for every separable Banach space $X$ with non-separable dual, the space $X^{**}$  contains an unconditional family of size $|X^{**}|$. As application  they proved

\begin{theorem}[Argyros-Dodos-Kanellopoulos]
\label{Dodos} Every infinite-dimensional dual Banach space has a separable
quotient.
\end{theorem}

\begin{corollary}
The space $\mathcal{L}\left( X,Y\right) $ of bounded linear operators
between Banach spaces $X$ and $Y$ equipped with the operator norm has a
separable quotient provided $Y\neq \left\{ \mathbf{0}\right\} $.
\end{corollary}
 Indeed, it
follows from the fact that $X^{\ast }$ is complemented in $\mathcal{L}\left(
X,Y\right) $, see Theorem \ref{Viola} below for details.

There are several equivalent conditions to Problem \ref{mazur}. Let's select
a few of them. For the equivalence (2) and (3) below (and applications) in
the class of locally convex spaces we refer the reader to \cite{sa1}, \cite%
{sa0} and \cite{salu}. The equivalence between (1) and (3) is due to
Saxon-Wilansky \cite{SW}. Recall that a locally convex space $E$ is \emph{%
barrelled} if every barrel (an absolutely convex closed and absorbing set in
$E$) is a neighborhood of zero. We refer also to \cite{benet} and \cite{JR}
for some partial results related to the next theorem.

\begin{theorem}[Saxon-Wilansky]
\label{Saxon}\emph{\ }The following assertions are equivalent for an
infinite-dimensional Banach space $X$.

\begin{enumerate}
\item $X$ contains a dense non-barrelled linear subspace.

\item $X$ admits a strictly increasing sequence of closed subspaces of $X$
whose union is dense in $X$.

\item $X^{\ast }$ admits a strictly decreasing sequence of weak*-closed
subspaces whose intersection consists only of the zero element.

\item $X$ has a separable quotient.
\end{enumerate}
\end{theorem}

\begin{proof}
We prove only the equivalence between $(2)$ and $(4)$ which holds for any
locally convex space $X$.

$(4)\Rightarrow (2)$ Note that every separable Banach space has the above
property (as stated in $(2)$) and this property is preserved under preimages
of surjective linear operators (which clearly are open maps).

$(2)\Rightarrow (4)$ Let $\left\{ X_{n}:n\in \mathbb{N}\right\} $ be such a
sequence. We may assume that $\dim (X_{n+1}/X_{n})\geq n$ for all $n\in
\mathbb{N}$. Let $x_{1}\in X_{2}\setminus X_{1}$. There exists $x_{1}^{\ast
} $ in $X^{\ast }$ such that $x_{1}^{\ast }x_{1}=1$ and $x_{1}^{\ast }$
vanishes on $X_{1}$. Assume that we have constructed $(x_{1},x_{1}^{\ast
}),\dots ,(x_{n},x_{n}^{\ast })$ in $X\times X^{\ast }$ with $x_{j}\in
X_{j+1}$ such that $x_{j}^{\ast }x_{j}=1$ and $x_{j}^{\ast }$ vanishes on $%
X_{j}$ for $1\leq j\leq n$. Choose $x_{n+1}\in X_{n+2}\cap (x_{1}^{\ast
})^{-1}(\mathbf{0})\cap \dots \cap (x_{n}^{\ast })^{-1}(\mathbf{0})\setminus
X_{n+1}.$ Then there exists $x_{n+1}^{\ast }\in X^{\ast }$ with $%
x_{n+1}^{\ast }x_{n+1}=1$ vanishing on $X_{n+1}$. Since $X_{n+1}\subseteq
\mathrm{span}\left\{ x_{1},x_{2},\dots x_{n}\right\} +\bigcap_{k=1}^{\infty
}\ker x_{k}^{\ast },n\in \mathbb{N},$ we conclude that $\mathrm{span}%
\{x_{n}:n\in \mathbb{N}\}+\bigcap_{n}\ker x_{n}^{\ast }$ is dense in $X$.
Then $X/Y$ is separable for $Y:=\bigcap_{n}\ker x_{n}^{\ast }$.
\end{proof}

Particularly every Banach space whose weak*-dual is separable has a
separable quotient. Theorem \ref{Saxon} applies to show that every infinite
dimensional WCG Banach space has a separable quotient. Indeed, if $X$ is
reflexive we apply Theorem \ref{Dodos}. If $X$ is not reflexive, choose a
weakly compact absolutely convex set $K$ in $X$ such that $\overline{\mathrm{%
span}\left( K\right) }=X$. Since $K$ is a barrel of $Y:=\mathrm{span}\left(
K\right) $ and $X$ is not reflexive, it turns out that $Y$ is a dense
non-barrelled linear subspace of $X$.

The class of WCG Banach spaces, introduced in \cite{amir}, provides a quite
successful generalization of reflexive and separable spaces. As proved in %
\cite{amir} there are many bounded projection operators with separable
ranges on such spaces, so many separable complemented subspaces. For
example, if $X$ is WCG and $Y$ is a separable subspace, there exists a
separable closed subspace $Z$ with $Y\subseteq Z\subseteq X$ together with a
contractive projection. This shows that every infinite-dimensional WCG
Banach space admits many separable complemented subspaces, so separable
quotients.

The Josefson-Nissenzweig theorem states that the dual of any
infinite-dimensional Banach space contains a \emph{normal sequence},
i.\thinspace e., a normalized weak*-null sequence \cite{nis}.

Recall (cf. \cite{sliwa}) that a sequence $\left\{ y_{n}^{\ast }\right\} $
in the sphere $S(X^{\ast })$ of $X^{\ast }$ is \emph{strongly normal} if the
subspace $\{x\in X:\sum_{n=1}^{\infty }|y_{n}^{\ast }x|<\infty \}$ is dense
in $X$. Clearly every strongly normal sequence is normal. Having in mind
Theorem \ref{sliwa}, the following question is of interest.

\begin{problem} [\'Sliwa]
\label{sliwa1} Does every normal sequence in $X^{*}$ contains a strongly
normal subsequence?
\end{problem}

By \cite[Theorem 1]{sliwa}, any strongly normal sequence in $X^{\ast }$
contains a subsequence $\left\{ y_{n}\right\} $ which is a Schauder basic
sequence in the weak*-topology, i.\thinspace e., $\{y_{n}\}$ is a Schauder
basis in its closed linear span in the weak$^{\ast }$-topology. Conversely,
any normalized Schauder basic sequence in $(X^{\ast },w^{\ast })$ is
strongly normal, \cite[Proposition 1]{sliwa}.

The following theorem from \cite{sliwa} exhibits a connection between these
concepts.

\begin{theorem}[\'Sliwa]
\label{sliwa} Let $X$ be an infinite-dimensional Banach space. The following
conditions are equivalent:

\begin{enumerate}
\item $X$ has a separable quotient.

\item $X^{*}$ has a strongly normal sequence.

\item $X^{\ast }$ has a basic sequence in the weak* topology.

\item $X^{\ast }$ has a pseudobounded sequence, i.\thinspace e., a sequence $%
\left\{ x_{n}^{\ast }\right\} $ in $X^{\ast }$ that is pointwise bounded on
a dense subspace of $X$ and $\sup_{n}\Vert x_{n}^{\ast }\Vert =\infty $.
\end{enumerate}
\end{theorem}

\begin{proof}
We prove only $(1)\Leftrightarrow (2)$ and $(1)\Leftrightarrow (4)$. The
equivalence between $(2)$ and $(3)$ follows from the preceding remark.

$(1)\Rightarrow (2)$ By Theorem \ref{Saxon} the space $X$ contains a dense
non-barrelled subspace. Hence there exists a closed absolutely convex set $D$
in $X$ such that $H:=\mathrm{span}\left( D\right) $ is a proper dense
subspace of $X$. For each $n\in \mathbb{N}$ choose $x_{n}\in X$ so that $%
\Vert x_{n}\Vert \leq n^{-2}$ and $x_{n}\notin D$. Then select $x_{n}^{\ast
} $ in $X^{\ast }$ such that $x_{n}^{\ast }x_{n}>1$ and $|x_{n}^{\ast
}x|\leq 1 $ for all $x\in D$. Set $y_{n}^{\ast }:=\Vert x_{n}^{\ast }\Vert
^{-1}x_{n}^{\ast }$ for all $n\in \mathbb{N}$. Since $\Vert x_{n}^{\ast
}\Vert \geq n^{2}$, $y_{n}^{\ast }\in S(X^{\ast })$ and $\sum_{n=1}^{\infty
}|y_{n}^{\ast }x|<\infty $ for all $x\in D$, the sequence $\left\{
y_{n}^{\ast }\right\} $ is as required.

$(2)\Rightarrow (1)$ Assume that $X^{\ast }$ contains a strongly normal
sequence $\left\{ y_{n}^{\ast }\right\} $. By \cite[Theorem 1]{sliwa} (see
also \cite[Theorem III.1 and Remark III.1]{JR}) there exists a subsequence
of $\left\{ y_{n}^{\ast }\right\} $ which is a weak*-basic sequence in $%
X^{\ast }$. This implies that $X$ admits a strictly increasing sequence of
closed subspaces whose union is dense in $X$. Indeed, since $(X^{\ast
},w^{\ast })$ contains a basic sequence, and hence there exists a strictly
decreasing sequence $\left\{ U_{n}:n\in \mathbb{N}\right\} $ of closed
subspaces in $(X^{\ast },w^{\ast })$ with $\bigcap_{n=1}^{\infty }U_{n}=\{%
\mathbf{0}\}$, the space $X$ has a sequence as required. This provides a
biorthogonal sequence as in the proof of $(2)\Rightarrow (3)$, Theorem \ref%
{Saxon}.

$(4)\Rightarrow (1)$ Let $\left\{ y_{n}^{\ast }\right\} $ be a pseudobounded
sequence in $X^{\ast }$. Set $Y:=\{x\in X:\sup_{n\in \mathbb{N}}|y_{n}^{\ast
}x|$ $<\infty \}$. The Banach-Steinhaus theorem applies to deduce that $Y$
is a proper and dense subspace of $X$. Note that $Y$ is not barrelled since $%
V:=\{x\in X:\sup_{n=1}^{\infty }|y_{n}^{\ast }x|\leq 1\}$ is a barrel in $H$
which is not a neighborhood of zero in $H$. Now apply Theorem \ref{Saxon}.

$(1)\Rightarrow (4)$ Assume $X$ contains a dense non-barrelled subspace $Y$.
Let $W$ be a barrel in $Y$ which is not a neighborhood of zero in $Y$. If $V$
is the closure of $W$ in $X$, the linear span $H$ of $V$ is a dense proper
subspace of $X$. So, for every $n\in \mathbb{N}$ there is $x_{n}\in
X\setminus V$ with $\Vert x_{n}\Vert \leq n^{-2}.$ Choose $z_{n}^{\ast }\in
X^{\ast }$ so that $|z_{n}^{\ast }x_{n}|>1$ and $|z_{n}^{\ast }x|=1$ for all
$x\in V$ and $n\in \mathbb{N}$. Then $\Vert z_{n}^{\ast }\Vert \geq n^{2}$
and $\sup_{n}|z_{n}^{\ast }x|<\infty $ for $x\in H$.
\end{proof}

A slightly stronger property than that of condition (2) of Theorem \ref%
{sliwa} is considered in the next proposition. As shown in the proof, this
property turns out to be equivalent to the presence in $X^{\ast }$ of a
basic sequence equivalent to the unit vector basis of $c_{0}$.

\begin{proposition}
\label{stronger} An infinite-dimensional Banach space $X$ has a quotient
isomorphic to $\ell _{1}$ if and only if $X^{\ast }$ contains a normal
sequence $\left\{ y_{n}^{\ast }\right\} $ such that $\sum_{n}|y_{n}^{\ast
}x|<\infty $ for all $x\in X$.
\end{proposition}

\begin{proof}
If there is a bounded linear operator $Q$ from $X$ onto $\ell _{1}$, its
adjoint map fixes a sequence $\left\{ x_{n}^{\ast }\right\} $ in $X^{\ast }$
such that the formal series $\sum_{n=1}^{\infty }x_{n}^{\ast }$ is weakly
unconditionally Cauchy and $\inf_{n\in \mathbb{N}}\left\| x_{n}^{\ast
}\right\| >0$. Setting $y_{n}^{\ast }:=\Vert x_{n}^{\ast }\Vert
^{-1}x_{n}^{\ast }$ for each $n\in \mathbb{N}$, the sequence $\left\{
y_{n}^{\ast }\right\} $ is as required. Conversely, if there is a normal
sequence $\left\{ y_{n}^{\ast }\right\} $ like that of the statement, it
defines a weak* Cauchy series in $X^{\ast }$. Since the series $%
\sum_{n=1}^{\infty }y_{n}^{\ast }$ does not converge in $X^{\ast }$,
according to \cite[Chapter V, Corollary 11]{Di} the space $X^{\ast }$ must
contain a copy of $\ell _{\infty }$. Consequently, $X$ has a complemented
copy of $\ell _{1}$ by \cite[Chapter V, Theorem 10]{Di}.
\end{proof}

We refer to the following large class of Banach spaces, for which Problem %
\ref{sliwa1} has a positive answer.

\begin{theorem}[\'Sliwa]
If $X$ is an infinite-dimensional WCG Banach space, every normal sequence in
$X^{*}$ contains a strongly normal subsequence.
\end{theorem}

An interesting consequence of Theorem \ref{Saxon} is that `small' Banach
spaces always have a separable quotient. We present another proof, different
from the one presented in \cite[Theorem 3]{SSR}, which depends on the
concept of strongly normal sequences.

\begin{corollary}[Saxon--Sanchez Ruiz]
\label{sax} If the density character $d\left( X\right) $ of a Banach space $%
X $ satisfies that $\aleph _{0}\leq d\left( X\right) <\mathfrak{b}$ then $X$
has a separable quotient.
\end{corollary}

Recall that the \textit{density character} of a Banach space $X$ is the
smallest cardinal of the dense subsets of $X$. The \textit{bounding cardinal}
$\mathfrak{b}$ is referred to as the minimum size for an unbounded subset of
the preordered space $\left( \mathbb{N}^{\mathbb{N}},\leq ^{\ast }\right) $,
where $\alpha \leq ^{\ast }\beta $ stands for the \textit{eventual dominance
preorder}$,$ defined so that $\alpha \leq ^{\ast }\beta $ if the set $%
\left\{ n\in \mathbb{N}:\alpha \left( n\right) >\beta \left( n\right)
\right\} $ is finite. So we have $\mathfrak{b}:=\inf \{|F|:F\subseteq
\mathbb{N}^{\mathbb{N}},\,\forall \alpha \in \mathbb{N}^{\mathbb{N}%
}\,\,\exists \beta \in F\,\ $with$\,\,\alpha <^{\ast }\beta \}$. It is well
known that $\mathfrak{b}$ is a regular cardinal and $\aleph _{0}<\mathfrak{b}%
\leq \mathfrak{c}$. It is consistent that $\mathfrak{b}=\mathfrak{c}>\aleph
_{1}$; indeed, Martin's Axiom implies that $\mathfrak{b}=\mathfrak{c}$.

\begin{proof}[Proof of Corollary \ref{sax}]
Assume $X$ has a dense subset $D$ of cardinality less than $\mathfrak{b}$.
We show that $X^{\ast }$ has a \emph{strongly normal sequence} and then we
apply Theorem \ref{sliwa}. Choose a normalized weak*-null sequence $\left\{
y_{n}^{\ast }\right\} $ in $X^{\ast }$. For $x\in D$ choose $\alpha _{x}\in
\mathbb{N}^{\mathbb{N}}$ such that for each $n\in \mathbb{N}$ and every $%
k\geq \alpha _{x}(n)$ one has $|y_{k}^{\ast }x|<2^{-n}$. Then $%
\sum_{n}|y_{\beta (n)}^{\ast }x|<\infty $ if $\alpha _{x}\leq ^{\ast }\beta $%
. Choose $\gamma \in \mathbb{N}^{\mathbb{N}}$ with $\alpha _{x}\leq ^{\ast
}\gamma $ for each $x\in D$. Then the sequence $\{y_{\gamma (n)}^{\ast }\}$
is strongly normal and Theorem \ref{sliwa} applies.
\end{proof}

For `large' Banach spaces we note the following interesting result \cite{Tod}.
\begin{theorem}[Todorcevic]
Under Martin's maximal axiom every
space of density character $\aleph _{1}$ has a quotient space with an
uncountable monotone Schauder basis, and thus a separable quotient.
\end{theorem}

Another line of research related to Problem \ref{mazur} deals with those
Banach spaces which contain complemented copies of concrete separable
sequence spaces. Recall the following important result found in Mujica's
survey paper (see \cite[Theorem 4.1]{Mu}). We need the following result due
to Rosenthal, see \cite[Corollary 1.6, Proposition 1.2]{Ro}.

\begin{lemma}
\label{Ro} Let $X$ be a Banach space such that $X^{*}$ contains an
infinite-dimensional reflexive subspace $Y$. Then $X$ has a quotient
isomorphic to $Y^{*}$. Consequently $X$ has a separable quotient.
\end{lemma}

\begin{proof}
Let $Q: X\rightarrow Y^{*}$ be defined by $Qx(y) = y(x)$ for $y\in Y$ and $%
x\in X$. Let $j: Y\rightarrow X^{*}$ and $\phi_{X}: X\rightarrow X^{**}$ be
the inclusion maps. Clearly $Q= j^{*}\circ\phi_{X}$ and $Q^{*}=\phi_{X}^{*}%
\circ j^{**}$. Since $Y$ is reflexive, $Q^{*}$ is an embedding map and
consequently $Q$ is surjective.
\end{proof}

\begin{theorem}[Mujica]
\label{Mujica} If $X$ is a Banach space that contains an isomorphic copy of $%
\ell _{1}$, then $X$ has a quotient isomorphic to $\ell _{2}$.
\end{theorem}

\begin{proof}
If $X$ contains a copy of $\ell _{1}$, the dual space $X^{\ast }$ contains a
copy of $L_{1}[0,1]$, see \cite{Di}. It is well known that the space $%
L_{1}[0,1]$ contains a copy of $\ell _{2}$. We apply Lemma \ref{Ro}.
\end{proof}

Concerning copies of $\ell _{1}$, let us recall that from classic
Rosenthal-Dor's $\ell _{1}$-dichotomy \cite[Chapter 11]{Di} one easily gets
the following general result.

\begin{theorem}
If $X$ is a non-reflexive weakly sequentially complete Banach space, then $X$
contains an isomorphic copy of $\ell _{1}$.
\end{theorem}

The previous results suggest also the following

\begin{problem}
Describe a possibly large class of non-reflexive Banach spaces $X$ not
containing an isomorphic copy of $\ell _{1}$ and having a separable quotient.
\end{problem}

We may summarize this section with the following

\begin{corollary}
\label{sep} Let $X$ be an infinite-dimensional Banach space. Assume that
either $X$ or $X^{*}$ contains an isomorphic copy of $c_{0}$ or either $X$
or $X^{*}$ contains an isomorphic copy of $\ell_{1}$. Then $X$ has separable
quotient.
\end{corollary}

It is interesting to remark that there exists an infinite-dimensional
separable Banach space $X$ such that neither $X$ nor $X^{\ast }$ contains a
copy of $c_{0}$, $\ell _{1}$ or an infinite-dimensional reflexive subspace
(see \cite{Go}). This apparently shows how difficult would be any approach  producing possible example of a Banach space without separable quotient.

We refer to \cite{kaple} for several results (and many references)
concerning $X$ not containing an isomorphic copy of $\ell _{1}$.

From now onwards, unless otherwise stated, $X$ is an infinite-dimensional
Banach space over the field $\mathbb{K}$ of real or complex numbers, as well
as all linear spaces we shall consider. Every measurable space $\left(
\Omega ,\Sigma \right) $, as well as every measure space $\left( \Omega
,\Sigma ,\mu \right) $, are supposed to be non trivial, i.\thinspace e.,
there are in $\Sigma $ infinitely many pairwise disjoints sets (of finite
positive measure). If either $X$ contains or does not contain an isomorphic
copy of a Banach space $Z$ we shall frequently write $X\supset Z$ or $%
X\not\supset Z$, respectively.

\section{Weak* compactness of $B_{X^{\ast }}$ and separable quotients}

In many cases the separable quotient problem is related to the
weak*-compactness of the dual unit ball, as the following theorem shows.

\begin{theorem}
\label{Juliet}Let $X$ be a Banach space and let $B_{X^{\ast }}\left( \mathrm{%
weak}^{\ast }\right) $ be the dual unit ball equipped with the weak*-
topology.

\begin{enumerate}
\item If $B_{X^{\ast }}\left( \mathrm{weak}^{\ast }\right) $ is not
sequentially compact, then $X$ has a separable quotient which is either
isomorphic to $c_{0}$ or to $\ell _{2}$.

\item If $B_{X^{\ast }}\left( \mathrm{weak}^{\ast }\right) $ is sequentially
compact, then $X$ has a copy of $c_{0}$ if and only if it has a complemented
copy of $c_{0}$.
\end{enumerate}
\end{theorem}

\proof%
(Sketch) For the first case, if $B_{X^{\ast }}\left( \mathrm{weak}^{\ast
}\right) $ is not sequentially compact, according to the classic
Hagler-Johnson theorem \cite[Corollary 1]{HJ}, $X$ either has a quotient
isomorphic to $c_{0}$ or $X$ contains a copy of $\ell _{1}$. The later case
implies that the dual space $X^{\ast }$ of $X$ contains a copy of $%
L_{1}[0,1] $, so that $X^{\ast }$ contains a copy of $\ell _{2}$. Hence $X$
has a quotient isomorphic to $\ell _{2}$.

The second statement follows from \cite{Em}, where the Gelfand-Phillips
property is used.~For a direct proof~we refer the reader~to \cite[Theorem
4.1]{Fe2}. We provide a brief account of the argument. Let $\left\{
x_{n}\right\} $ be a normalized basic sequence in $X$ equivalent to the unit
vector basis $\left\{ e_{n}\right\} $ of $c_{0}$ and let $\left\{
x_{n}^{\ast }\right\} $ denote the sequence of coordinate functionals of $%
\left\{ x_{n}\right\} $ extended to $X$ via Hahn-Banach's theorem. If $K>0$
is the basis constant of $\left\{ x_{n}\right\} $ then $\left\| x_{n}^{\ast
}\right\| \leq 2K$, so that $x_{n}^{\ast }\in 2KB_{X^{\ast }}$ for every $%
n\in \mathbb{N}$. Since $B_{X^{\ast }}\left( \mathrm{weak}^{\ast }\right) $
is sequentially compact, there is a subsequence $\left\{ z_{n}^{\ast
}\right\} $ of $\left\{ x_{n}^{\ast }\right\} $ that converges to a point $%
z^{\ast }\in X^{\ast }$ under the weak*-topology. Let $\left\{ z_{n}\right\}
$ be the corresponding subsequence of $\left\{ x_{n}\right\} $, still
equivalent to the unit vector basis of $c_{0}$, and let $F$ be the closed
linear span of $\left\{ z_{n}:n\in \mathbb{N}\right\} $. For each $n\in
\mathbb{N}$ define the linear functional $u_{n}:X\rightarrow \mathbb{K}$ by $%
u_{n}\left( x\right) =(z_{n}^{\ast }-z^{\ast })x$, so that $\left|
u_{n}\left( x\right) \right| \leq 4K\left\| x\right\| $ for each $n\in
\mathbb{N}$. Since $u_{n}\left( x\right) \rightarrow 0$ for all $x\in X$,
the linear operator $P:X\rightarrow F$ given by $Px=\sum_{n=1}^{\infty
}u_{n}\left( x\right) z_{n}$ is well defined. Due to the formal series $%
\sum_{n=1}^{\infty }z_{n}$ is weakly unconditionally Cauchy, there is a
constant $C>0$ such that $\left\| Px\right\| \leq 4CK\left\| x\right\| $.
Now the fact that $z_{n}^{\ast }y\rightarrow 0$ for each $y\in F$ means that
$z^{\ast }\in F^{\bot }$, which implies that $Pz_{j}=z_{j}$ for each $j\in
\mathbb{N}$. Thus $P$ is a bounded linear projection operator from $X$ onto $%
F$.
\endproof%

It can be easily seen that if $X^{\ast }$ contains an isomorphic copy of $%
\ell _{1}$ but $X$ does not, then $X$ has a quotient isomorphic to $c_{0}$.
If $X^{\ast }$ has a copy of $\ell _{1}$, according to the first part of
Theorem \ref{Juliet}, then $X$ has a separable quotient isomorphic to $c_{0}$
or $\ell _{2}$.

On the other hand, the first statement of the previous theorem also implies
that each Banach space that contains an isomorphic copy of $\ell _{1}\left(
\mathbb{R}\right) $ has a quotient isomorphic to $c_{0}$ or $\ell _{2}$.
Particularly, each Banach space $X$ containing an isomorphic copy of $\ell
_{\infty }$ enjoys this property. However, since $\ell _{\infty }$ is an
injective Banach space and it has a separable quotient isomorphic to $\ell
_{2}$ (as follows, for instance, from Theorem \ref{Mujica}), one derives
that $X$ has a separable quotient isomorphic to $\ell _{2}$ provided $X$
contains an isomorphic copy of $\ell _{\infty }$. Useful characterizations
of Banach spaces containing a copy of $\ell _{\infty }$ can be found in the
classic paper \cite{Ro2}. It is also shown in \cite{Ta} that $\ell _{\infty
} $ is a quotient of a Banach space $X$ if and only if $B_{X^{\ast }}$
contains a weak*-homeomorphic copy of $\beta \mathbb{N}$. Hence such a space
$X$ has in particular a separable quotient isomorphic to $\ell _{2}$.

The class of Banach spaces for which $B_{X^{\ast }}\left( \mathrm{weak}%
^{\ast }\right) $ is sequentially compact is rich. This happens, for example
if $X$ is a WCG Banach space. Of course, no WCG Banach space contains a copy
of $\ell _{\infty }$. Another class with weak* sequentially compact dual
balls is that of \textit{Asplund }spaces. Note that the second statement of
Theorem \ref{Juliet} applies in particular to each Banach space whose
weak*-dual unit ball is Corson's (a fact first observed in \cite{Mo}) since,
as is well-known, each Corson compact is Fr\'{e}chet-Urysohn. So, one has
the following corollary, where a Banach space $X$ is called \textit{weakly
Lindel\"{o}f determined} (WLD for short) if there is a set $M\subseteq X$
with $\overline{\mathrm{span}\left( M\right) }=X$ enjoying the property that
for each $x^{\ast }\in X^{\ast }$ the set $\left\{ x\in M:x^{\ast }x\neq
0\right\} $ is countable.

\begin{corollary}
If $X$ is a WLD Banach space, then $X$ contains a complemented copy of $%
c_{0} $ if and only if it contains a copy of $c_{0}$.
\end{corollary}

\proof%
If $X$ is a WLD Banach space, the dual unit ball $B_{X^{\ast }}\left(
\mathrm{weak}^{\ast }\right) $ of $X$ is Corson (see \cite[Proposition 1.2]%
{AM}), so the second statement of Theorem \ref{Juliet} applies.
\endproof%

If $K$ is an infinite Gul'ko compact space, then $C\left( K\right) $ is
weakly countable determined (see \cite{AN}), hence WLD. Since $C\left(
K\right) $ has plenty of copies of $c_{0}$, it must have many complemented
copies of $c_{0}$. It must be pointed out that if $K$ is Corson compact then
$C\left( K\right) $ need not be WLD. On the other hand, if a Banach space $%
X\supset c_{0}$ then $X^{\ast }\supset \ell _{1}$, so surely $X$ has $c_{0}$
or $\ell _{2}$ as a quotient (a general characterization of Banach spaces
containing a copy of $c_{0}$ is provided in \cite{Ro3}). This fact can be
sharpened, as the next corollary shows.

\begin{corollary}
If a Banach space $X$ contains a copy of $c_{0}$, then $X$ has either an
infinite-dimensional separable quotient isomorphic to $c_{0}$ or $\ell _{2}$%
, or a complemented copy of $c_{0}$.
\end{corollary}

\proof%
If $B_{X^{\ast }}\left( \mathrm{weak}^{\ast }\right) $ is not sequentially
compact, $X$ has a separable quotient isomorphic to $c_{0}$ or $\ell _{2}$
as a consequence of the first part of Theorem \ref{Juliet}. If $B_{X^{\ast
}}\left( \mathrm{weak}^{\ast }\right) $ is sequentially compact, by the
second part $X$ has a complemented copy of $c_{0}$.
\endproof%

\begin{corollary}
\emph{(cf. \cite{La} and \cite{Ro})} If $K$ is an infinite compact Hausdorff
space, then $C\left( K\right) $ always has a quotient isomorphic to $c_{0}$
or $\ell _{2}$. In case that $K$ is scattered, then $c_{0}$ embeds in $%
C\left( K\right) $ complementably.
\end{corollary}

\proof%
The first statement is clear. The second is due to in this case $C\left(
K\right) $ is an Asplund space (see \cite[Theorem 296]{HHZ}).
\endproof%

An extension of the previous corollary to all barrelled spaces $C_{k}(X)$
with the compact-open topology has been obtained in \cite{sa1}.

\section{Separable quotients in tensor products}

We quote three classic results about the existence of copies of $c_{0}$, $%
\ell _{\infty }$ and $\ell _{1}$ in injective and projective tensor products
which will be frequently used henceforth and provide a result concerning the
existence of a separable quotient in $X\,\widehat{\otimes }_{\pi }Y$. We
complement these classic facts with other results of our own. In the
following theorem $c_{00}$ stands for the linear subspace of $c_{0}$
consisting of all those sequence of finite range.

\begin{theorem}
\label{Freniche}\emph{(cf. \cite[Theorem 2.3]{Fre})} Let $X$ be an
infinite-dimensional normed space and let $Y$ be a Hausdorff locally convex
space. If $Y\supset c_{00}$ then $X\,\widehat{\otimes }_{\varepsilon }Y$
contains a complemented subspace isomorphic to $c_{0}$.
\end{theorem}

Particularly, if $X$ and $Y$ are infinite-dimensional Banach spaces and $%
X\supset c_{0}$ or $Y\supset c_{0}$, then $X\,\widehat{\otimes }%
_{\varepsilon }Y$ contains a complemented copy of $c_{0}$, (cf. \cite{SS2}).
On the other hand, if either $X\supset \ell _{\infty }$ or $Y\supset \ell
_{\infty }$ then $X\,\widehat{\otimes }_{\varepsilon }Y\supset \ell _{\infty
}$ and consequently $X\,\widehat{\otimes }_{\varepsilon }Y$ also has a
separable quotient isomorphic to $\ell _{2}$. If $X\,\widehat{\otimes }%
_{\varepsilon }Y\supset \ell _{\infty }$, the converse statement also holds,
as the next theorem asserts.

\begin{theorem}
\label{Drewnowski}\emph{(cf. \cite[Corollary 2]{Dr2})} Let $X$ and $Y$ be
Banach spaces. $X\,\widehat{\otimes }_{\varepsilon }Y\supset \ell _{\infty }$
if and only if $X\supset \ell _{\infty }$ or $Y\supset \ell _{\infty }$.
\end{theorem}

This also implies that if $X\,\widehat{\otimes }_{\varepsilon }Y\supset \ell
_{\infty }$ then $c_{0}$ embeds complementably in $X\,\widehat{\otimes }%
_{\varepsilon }Y$. Concerning projective tensor products, we have the
following well-known fact.

\begin{theorem}
\label{Bombal}\emph{(cf. \cite[Corollary 2.6]{BFV})} Let $X$ and $Y$ be
Banach spaces. If both $X\supset \ell _{1}$ and $Y\supset \ell _{1}$, then $%
X\,\widehat{\otimes }_{\pi }Y$ has a complemented subspace isomorphic to $%
\ell _{1}$.
\end{theorem}

Next we observe that if $X\,\widehat{\otimes }_{\varepsilon }Y$ is not a
quotient of $X\,\widehat{\otimes }_{\pi }Y$, then $X\,\widehat{\otimes }%
_{\varepsilon }Y$ has a separable quotient.

\begin{theorem}
\label{Miranda} Let $J:X\otimes _{\pi }Y\rightarrow X\otimes _{\varepsilon
}Y $ be the identity map and consider the continuous linear extension $%
\widetilde{J}:X\,\widehat{\otimes }_{\pi }Y\rightarrow X\,\widehat{\otimes }%
_{\varepsilon }Y$. If $\widetilde{J}$ is not a quotient map, then $X\,%
\widehat{\otimes }_{\varepsilon }Y$ has a separable quotient.
\end{theorem}

\proof%
Observe that $X\,\otimes _{\varepsilon }Y\subset\text{Im}\widetilde{J}%
\subset X\,\widehat{\otimes }_{\varepsilon }Y$. Two cases are in order.

Assume first that $X\,\otimes _{\varepsilon }Y$ is a barrelled space. In
this case, since $X\,\otimes _{\varepsilon }Y$ is dense in $\text{Im}%
\widetilde{J}$, we have that the range space $\text{Im}\widetilde{J}$ is a
barrelled subspace of $X\,\widehat{\otimes }_{\varepsilon }Y$. Given that
the graph of $\widetilde{J}$ is closed in $(X\,\widehat{\otimes }_{\pi
}Y)\times (X\,\widehat{\otimes }_{\varepsilon }Y)$ and $\text{Im}\widetilde{J%
}$ is barrelled, it follows from \cite[Theorem 19]{Va} that $\text{Im}%
\widetilde{J}$ is a closed subspace of $X\,\widehat{\otimes }_{\varepsilon
}Y $. Of course, this means that $\text{Im}\widetilde{J}=X\,\widehat{\otimes
}_{\varepsilon }Y$. Hence, the open map theorem shows that $\widetilde{J}$
is an open map from $X\,\widehat{\otimes }_{\pi }Y$ onto $X\,\widehat{%
\otimes }_{\varepsilon }Y$, so that $X\,\widehat{\otimes }_{\varepsilon }Y$
is a quotient of $X\,\widehat{\otimes }_{\pi }Y$.

Assume now that $X\,\otimes _{\varepsilon }Y$ is not barrelled. In this case
$X\,\otimes _{\varepsilon }Y$ is a non barrelled dense subspace of the
Banach space $X\,\widehat{\otimes }_{\varepsilon }Y$, so we may apply
Theorem \ref{Saxon} to get that $X\,\widehat{\otimes }_{\varepsilon }Y$ has
a separable quotient.
\endproof%
Recall that the dual of $X\otimes_{\pi}Y$ coincides with the space of
bounded linear operators from $X$ into $Y^{*}$, whereas the dual of $%
X\otimes_{\epsilon}T$ may be identified with the subspace of those operators
which are integral, see \cite[Section 3.5]{Ry}.

\begin{proposition}
Let $X$ and $Y$ be Banach spaces. If $X$ has the bounded approximation
property and there is a bounded linear operator $T:X\rightarrow Y^{\ast }$
which is not integral, then $X\,\widehat{\otimes }_{\varepsilon }Y$ has a
separable quotient.
\end{proposition}

\proof%
Since does exist a bounded not integral linear operator between $X$ and $%
Y^{*}$, the $\pi$-topology and $\epsilon$-topology does not coincide on $X
\otimes Y$, see \cite{Ry}. Assume $X\otimes _{\varepsilon }Y$ is barrelled.
Since $X$ has the bounded approximation property, \cite[Theorem]{Bon}
applies to get that $X\otimes _{\varepsilon }Y=X\otimes _{\pi }Y$, which
contradicts the assumption that $\left( X\otimes _{\varepsilon }Y\right)
^{\ast }\neq \left( X\otimes _{\pi }Y\right) ^{\ast }$. Thus $X\otimes
_{\varepsilon }Y$ must be a non barrelled dense linear subspace of $X\,%
\widehat{\otimes }_{\varepsilon }Y$, which according to Theorem \ref{Saxon}
ensures that $X\,\widehat{\otimes }_{\varepsilon }Y$ has a separable
quotient.
\endproof%

For the next theorem, recall that a Banach space $X$ is called \textit{%
weakly countably determined} (WCD for short) if $X\left( \mathrm{weak}%
\right) $ is a Lindel\"{o}f $\Sigma $-space.

\begin{theorem}
\label{Minerva}Let $X$ and $Y$ be WCD Banach spaces. If $X\,\widehat{\otimes
}_{\varepsilon }Y\supset c_{0}$, then $c_{0}$ embeds complementably in $X\,%
\widehat{\otimes }_{\varepsilon }Y$.
\end{theorem}

\proof%
Since both $X$ and $Y$ are WCD Banach spaces, their dual unit balls $%
B_{X^{\ast }}\left( \mathrm{weak}^{\ast }\right) $ and $B_{X^{\ast }}\left(
\mathrm{weak}^{\ast }\right) $ are Gul'ko compact. Given that the countable
product of Gul'ko compact spaces is Gul'ko compact, the product space $%
K:=B_{X^{\ast }}\left( \mathrm{weak}^{\ast }\right) \times B_{X^{\ast
}}\left( \mathrm{weak}^{\ast }\right) $ is Gul'ko compact. Consequently $%
C\left( K\right) $ is a WCD Banach space, which implies in turn that its
weak*-dual unit ball $B_{C\left( K\right) ^{\ast }}$ is Gul'ko compact.
Particularly $B_{C\left( K\right) ^{\ast }}\left( \mathrm{weak}^{\ast
}\right) $ is angelic and consequently sequentially compact. Let $Z$ stand
for the isometric copy of $X\,\widehat{\otimes }_{\varepsilon }Y$ in $%
C\left( K\right) $ and $P$ for the isomorphic copy of $c_{0}$ in $Z$. From
the proof of the second statement of Theorem \ref{Juliet} it follows that $%
C\left( K\right) $ has a complemented copy $Q$ of $c_{0}$ \textit{contained}
in $P$. This implies that $Z$, hence $X\,\widehat{\otimes }_{\varepsilon }Y$%
, contains a complemented copy $Q$ of $c_{0}$.%
\endproof%

\section{Separable quotients in spaces of vector-valued functions}

If $\left( \Omega ,\Sigma ,\mu \right) $ is a non trivial arbitrary measure
space, we denote by $L_{p}\left( \mu ,X\right) $, $1\leq p\leq \infty $, the
Banach space of all $X$-valued $p$-Bochner $\mu $-integrable ($\mu $%
-essentially bounded when $p=\infty $) classes of functions equipped with
its usual norm. If $K$ is an infinite compact Hausdorff space, then $C\left(
K,X\right) $ stands for the Banach space of all continuous functions $%
f:K\rightarrow X$ equipped with the supremum norm. By $B\left( \Sigma
,X\right) $ we represent the Banach space of those bounded functions $%
f:\Omega \rightarrow X$ that are the uniform limit of a sequence of $\Sigma $%
-simple and $X$-valued functions, equipped with the supremum norm. The space
of \textit{all} $X$-valued bounded functions $f:\Omega \rightarrow X$
endowed with the supremum norm is written as $\ell _{\infty }\left( \Omega
,X\right) $. Clearly $\ell _{\infty }\left( X\right) =\ell _{\infty }\left(
\mathbb{N},X\right) $. By $\ell _{\infty }\left( \Sigma \right) $ we denote
the completion of the space $\ell _{0}^{\infty }\left( \Sigma \right) $ of
scalarly-valued $\Sigma $-simple functions, endowed with the supremum norm.

On the other hand, if $\left( \Omega ,\Sigma ,\mu \right) $ is a (complete)
finite measure space we represent by $P_{1}(\mu ,X)$ the normed space
consisting of all those [classes of] strongly $\mu $-measurable $X$-valued
Pettis integrable functions $f$ defined on $\Omega $ provided with the
semivariation norm
\begin{equation*}
\left\| f\right\| _{P_{1}\left( \mu ,X\right) }=\sup \left\{ \int_{\Omega
}\left| x^{\ast }f\left( \omega \right) \right| \,d\mu \left( \omega \right)
:x^{\ast }\in X^{\ast },\,\left\| x^{\ast }\right\| \leq 1\right\} .
\end{equation*}%
As is well known, in general $P_{1}(\mu ,X)$ is not a Banach space if $X$ is
infinite-dimensional, but it is always a barrelled space (see \cite[Theorem
2]{DFP} and \cite[Remark 10.5.5]{FLS}).

Our first result collects together a number of statements concerning Banach
spaces of vector-valued functions related to the existence of separable
quotients, most of them easily derived from well known facts relative to the
presence of complemented copies of $c_{0}$ and $\ell _{p}$ for $1\leq p\leq
\infty $. We denote by $ca^{+}\left( \Sigma \right) $ the set of positive
and finite measures on $\Sigma $.

\begin{theorem}
\label{Cordelia}The following statements on spaces of vector-valued
functions hold.

\begin{enumerate}
\item $C\left( K,X\right) $ always has a complemented copy of $c_{0}$.

\item $C\left( K,X\right) $ has a quotient isomorphic to $\ell _{1}$ if and
only if $X$ has $\ell _{1}$ as a quotient.

\item $L_{p}\left( \mu ,X\right) $, with $1\leq p<\infty $, has a
complemented copy of $\ell _{p}$. In particular, the vector sequence space $%
\ell _{p}\left( X\right) $ has a complemented copy of $\ell _{p}$.

\item $L_{p}\left( \mu ,X\right) $, with $1<p<\infty $, has a quotient
isomorphic to $\ell _{1}$ if and only if $X$ has $\ell _{1}$ as a quotient.
Particularly $\ell _{p}\left( X\right) $ has a quotient isomorphic to $\ell
_{1}$ if and only if the same happens to $X$.

\item $L_{\infty }\left( \mu ,X\right) $ has a quotient isomorphic to $\ell
_{2}$. Hence, so does $\ell _{\infty }\left( X\right) $.

\item If $\mu $ is purely atomic and $1\leq p<\infty $, then $L_{p}\left(
\mu ,X\right) $ has a complemented copy of $c_{0}$ if and only if $X$ has a
complemented copy of $c_{0}$. Particularly, the space $\ell _{p}\left(
X\right) $ has a complemented copy of $c_{0}$ if and only if so does $X$.

\item If $\mu $ is not purely atomic and $1\leq p<\infty $, then $%
L_{p}\left( \mu ,X\right) $ has complemented copy of $c_{0}$ if $X\supset
c_{0}$.

\item If $\mu \in ca^{+}\left( \Sigma \right) $ is purely atomic and $%
1<p<\infty $, then $L_{p}\left( \mu ,X\right) $ has a quotient isomorphic to
$c_{0}$ if and only if $X$ contains a quotient isomorphic to $c_{0}$.

\item If $\mu \in ca^{+}\left( \Sigma \right) $ is not purely atomic and $%
1<p<\infty $, then $L_{p}\left( \mu ,X\right) $ has a quotient isomorphic to
$c_{0}$ if and only if $X$ contains a quotient isomorphic to $c_{0}$ or $%
X\supset \ell _{1}$.

\item If $\mu $ is $\sigma $-finite, then $L_{\infty }\left( \mu ,X\right) $
has a quotient isomorphic to $\ell _{1}$ if and only if $\ell _{\infty
}\left( X\right) $ has $\ell _{1}$ as a quotient.

\item $B\left( \Sigma ,X\right) $ has a complemented copy of $c_{0}$ and a
quotient isomorphic to $\ell _{2}$.

\item $\ell _{\infty }\left( \Omega ,X\right) $ has a quotient isomorphic to
$\ell _{2}$.

\item If the cardinality of $\Omega $ is less than the first real-valued
measurable cardinal, then $\ell _{\infty }\left( \Omega ,X\right) $ has a
complemented copy of $c_{0}$ if and only if $X$ enjoys the same property.
Particularly, $\ell _{\infty }\left( X\right) $ contains a complemented copy
of $c_{0}$ if and only if $X$ enjoys the same property.

\item $c_{0}\left( X\right) $ has a complemented copy of $c_{0}$.

\item $\ell _{\infty }\left( X\right) ^{\ast }$ has a quotient isomorphic to
$\ell _{1}$.
\end{enumerate}
\end{theorem}

\proof%
Let us proceed with the proofs of the statements.

\begin{enumerate}
\item This well-know fact can be found in \cite[Theorem]{Ce} and %
\cite[Corollary 2.5]{Fre} (or in \cite[Theorem 3.2.1]{CM}).

\item This is because $C\left( K,X\right) $ contains a complemented copy of $%
\ell _{1}$ if and only if $X$ contains a complemented copy of $\ell _{1}$
(see \cite{SS} or \cite[Theorem 3.1.4]{CM}).

\item If $1\leq p<\infty $, each $L_{p}\left( \mu ,X\right) $ space contains
a norm one complemented isometric copy of $\ell _{p}$ (see \cite[Proposition
1.4.1]{CM}). For the second affirmation note that if $\left( \Omega ,\Sigma
,\mu \right) $ is a $\sigma $-finite purely atomic measure space, then $\ell
_{p}\left( X\right) =L_{p}\left( \mu ,X\right) $ isometrically.

\item If $1<p<\infty $ then $L_{p}\left( \mu ,X\right) $ contains
complemented copy of $\ell _{1}$ if and only if $X$ does (see \cite{Me2} or %
\cite[Theorem 4.1.2]{CM}).

\item The space $\ell _{\infty }$ is isometrically embedded in $L_{\infty
}\left( \mu \right) $, which is in turn isometric to a norm one complemented
subspace of $L_{\infty }\left( \mu ,X\right) $.

\item If $\mu $ is purely atomic, then $L_{p}\left( \mu ,X\right) $ contains
a complemented copy of $c_{0}$ if and only if $X$ has the same property.
This fact, discovered by F. Bombal in \cite{Bo}, can also be seen in %
\cite[Theorem 4.3.1]{CM}.

\item If $\mu $ is not purely atomic and $1\leq p<\infty $, according to %
\cite{Em1}, the only fact that $X\supset c_{0}$ implies that $L_{p}\left(
\mu ,X\right) $ contains a complemented copy of $c_{0}$.

\item If $\left( \Omega ,\Sigma ,\mu \right) $ is a purely atomic finite
measure space and $1<p<\infty $, the statement corresponds to the first
statement of \cite[Theorem 1.1]{DS}.

\item If $\left( \Omega ,\Sigma ,\mu \right) $ is a not purely atomic finite
measure space and $1<p<\infty $, the statement corresponds to the second
statement of \cite[Theorem 1.1]{DS}.

\item If $\left( \Omega ,\Sigma ,\mu \right) $ is a $\sigma $-finite measure
space, the existence of a complemented copy of $\ell _{1}$ in $L_{\infty
}\left( \mu ,X\right) $ is related to the local theory of Banach spaces, a
fact discovered by S. D\'{\i}az in \cite{Dia2}. The statement, in the way as
it has been formulated above, can be found in \cite[Theorem 5.2.3]{CM}.

\item Since $\ell _{0}^{\infty }\left( \Sigma ,X\right) =\ell _{0}^{\infty
}\left( \Sigma \right) \otimes _{\varepsilon }X$ and $X$ is
infinite-dimensional, then $\ell _{0}^{\infty }\left( \Sigma ,X\right) $ is
not barrelled by virtue of classic Freniche's theorem (see \cite[Corollar
1.5]{Fre}). Given that $\ell _{0}^{\infty }\left( \Sigma ,X\right) $ is a
non barrelled dense subspace of $B\left( \Sigma ,X\right) $, Theorem \ref%
{Saxon} guarantees that $B\left( \Sigma ,X\right) $ has in fact a separable
quotient. However, we can be more precise. Since $B\left( \Sigma ,X\right)
=\ell _{0}^{\infty }\left( \Sigma \right) \,\widehat{\otimes }_{\varepsilon
}X$ and $\ell _{0}^{\infty }\left( \Sigma \right) \supset c_{00}$ due to the
non triviality of the $\sigma $-algebra $\Sigma $, Theorem \ref{Freniche}
implies that $B\left( \Sigma ,X\right) $ contains a complemented copy of $%
c_{0}$. On the other hand, since $\ell _{\infty }$ is isometrically embedded
in $B\left( \Sigma ,X\right) $, it turns out that $\ell _{2}$ is a quotient
of $B\left( \Sigma ,X\right) $.

\item Clearly $\ell _{\infty }\left( \Omega ,X\right) \supset \ell _{\infty
}\left( \Omega \right) \supset \ell _{\infty }$ since the set $\Omega $ is
infinite.

\item This property can be found in \cite{LR}.

\item Just note that $c_{0}\left( X\right) =c_{0}\,\widehat{\otimes }%
_{\varepsilon }X$, so we may apply Theorem \ref{Freniche}.

\item It suffices to note that $\ell _{1}\left( X^{\ast }\right) $ is
linearly isometric to a complemented subspace of $\ell _{\infty }\left(
X\right) ^{\ast }$ (see \cite[Section 5.1]{CM}).
\endproof%
\end{enumerate}

\begin{remark}
$L_{p}\left( \mu ,X\right) $ if $1\leq p<\infty $, as well as $C\left(
K,X\right) $, need not contain a copy of $\ell _{\infty }$. \emph{By %
\cite[Theorem]{Me}, one has that }$L_{p}\left( \mu ,X\right) \supset \ell
_{\infty }$\emph{\ if and only if }$X\supset \ell _{\infty }$\emph{, whereas
}$C\left( K,X\right) \supset \ell _{\infty }$\emph{\ if and only if }$%
C\left( K\right) \supset \ell _{\infty }$\emph{\ or }$X\supset \ell _{\infty
}$\emph{, as shown in \cite[Corollary 3]{Dr2}.}
\end{remark}

\begin{remark}
Complemented copies of $c_{0}$ in $L_{\infty }\left( \mu ,X\right) $. \emph{%
If }$\left( \Omega ,\Sigma ,\mu \right) $\emph{\ is a }$\sigma $\emph{%
-finite measure, according to \cite[Theorem 1]{Dia} a necessary condition
for the space }$L_{\infty }\left( \mu ,X\right) $\emph{\ to contain a
complemented copy of }$c_{0}$\emph{\ is that }$X\supset c_{0}$\emph{. The
same happens with the space }$\ell _{\infty }\left( \Omega ,X\right) $\emph{%
\ (see \cite[Theorem 2.1 and Corollary 2.3]{Fe2}).}
\end{remark}

\begin{theorem}
\label{Vera}The following statements on the space $\widehat{P_{1}\left( \mu
,X\right) }$ hold.

\begin{enumerate}
\item If the finite measure space $\left( \Omega ,\Sigma ,\mu \right) $ is
not purely atomic, the Banach space $\widehat{P_{1}\left( \mu ,X\right) }$
has a separable quotient.

\item If the range of the positive finite measure $\mu $ is infinite and $%
X\supset c_{0}$ then $\widehat{P_{1}\left( \mu ,X\right) }$ has a
complemented copy of $c_{0}$.
\end{enumerate}
\end{theorem}

\proof%
Observe that $L_{1}\left( \mu \right) \,\widehat{\otimes }_{\pi
}X=L_{1}\left( \mu ,X\right) $ and $\widehat{P_{1}\left( \mu ,X\right) }%
=L_{1}\left( \mu \right) \,\widehat{\otimes }_{\varepsilon }X$
isometrically. On the other hand, from the algebraic viewpoint $L_{1}\left(
\mu ,X\right) $ is a linear subspace of $P_{1}\left( \mu ,X\right) $, which
is dense under the norm of $P_{1}\left( \mu ,X\right) $. If
\begin{equation*}
J:L_{1}\left( \mu \right) \otimes _{\pi }X\rightarrow L_{1}\left( \mu
\right) \otimes _{\varepsilon }X
\end{equation*}
is the identity map, $R$ a linear isometry from $L_{1}\left( \mu ,X\right) $
onto $L_{1}\left( \mu \right) \,\widehat{\otimes }_{\pi }X$ and $S$ a linear
isometry from $L_{1}\left( \mu \right) \,\widehat{\otimes }_{\varepsilon }X$
onto $\widehat{P_{1}\left( \mu ,X\right) }$, the mapping
\begin{equation*}
S\circ \widetilde{J}\circ R:L_{1}\left( \mu ,X\right) \rightarrow \widehat{%
P_{1}\left( \mu ,X\right) },
\end{equation*}
where $\widetilde{J}$ denotes the (unique) continuous linear extension of $J$
to $L_{1}\left( \mu \right) \,\widehat{\otimes }_{\pi }X$, coincides with
the natural inclusion map $T$ of $L_{1}\left( \mu ,X\right) $ into $%
P_{1}\left( \mu ,X\right) $ over the dense subspace of $L_{1}\left( \mu
,X\right) $ consisting of the $X$-valued (classes of) $\mu $-simple
functions, which implies that $S\circ \widetilde{J}\circ R=T$. Since $X$ is
infinite-dimensional and $\mu $ is not purely atomic, the space $P_{1}\left(
\mu ,X\right) $ is not complete \cite{To}. So necessarily we have that $%
\text{Im} T\neq \widehat{P_{1}\left( \mu ,X\right) }$. This implies in
particular that $\text{Im}\widetilde{J}\neq L_{1}\left( \mu \right) \,%
\widehat{\otimes }_{\pi }X$. According to Theorem \ref{Miranda}, this means
that $\widehat{P_{1}\left( \mu ,X\right) }=L_{1}\left( \mu \right) \,%
\widehat{\otimes }_{\varepsilon }X$ has a separable quotient.

The proof of the second statement can be found in \cite[Corollary 2]{Fre2}.
\endproof%

\section{Separable quotients in spaces of linear operators}

If $Y$ is also a Banach space, let us denote by $\mathcal{L}\left(
X,Y\right) $ the Banach space of all bounded linear operators $%
T:X\rightarrow Y$ equipped with the operator norm $\left\| T\right\| $. By $%
\mathcal{K}\left( X,Y\right) $ we represent the closed linear subspace of $%
\mathcal{L}\left( X,Y\right) $ consisting of all those compact operators. We
design by $\mathcal{L}_{w^{\ast }}\left( X^{\ast },Y\right) $ the closed
linear subspace of $\mathcal{L}\left( X^{\ast },Y\right) $ formed by all
weak*-weakly continuous operators and by $\mathcal{K}_{w^{\ast }}\left(
X^{\ast },Y\right) $ the closed linear subspace of $\mathcal{K}\left(
X^{\ast },Y\right) $ consisting of all weak*-weakly continuous operators.
The closed subspace of $\mathcal{L}\left( X,Y\right) $ consisting of weakly
compact linear operators is designed by $\mathcal{W}\left( X,Y\right) $. It
is worthwhile to mention that $\mathcal{L}_{w^{\ast }}\left( X^{\ast
},Y\right) =\mathcal{L}_{w^{\ast }}\left( Y^{\ast },X\right) $
isometrically, as well as $\mathcal{K}_{w^{\ast }}\left( X^{\ast },Y\right) =%
\mathcal{K}_{w^{\ast }}\left( Y^{\ast },X\right) $, by means of the linear
mapping $T\mapsto T^{\ast }$. The Banach space of nuclear operators $%
T:X\rightarrow Y$ equipped with the so-called nuclear norm $\left\|
T\right\| _{N}$ is denoted by $\mathcal{N}\left( X,Y\right) $. Let us recall
that $\left\| T\right\| \leq \left\| T\right\| _{N}$. Classic references for
this section are the monographs \cite{Ja} and \cite{Ko}.

The first statement of Theorem \ref{Viola} answers a question of Prof. T.
Dobrowolski posed during the 31st Summer Conference on Topology and its
Applications at Leicester (2016).

\begin{theorem}
\label{Viola}The following conditions on $\mathcal{L}\left( X,Y\right) $
hold.

\begin{enumerate}
\item If $Y\neq \left\{ \mathbf{0}\right\} $, then $\mathcal{L}\left(
X,Y\right) $ always has a separable quotient.

\item If $X^{\ast }\supset c_{0}$ or $Y\supset c_{0}$, then $\mathcal{L}%
\left( X,Y\right) $ has a quotient isomorphic to $\ell _{2}$.

\item If $X^{\ast }\supset \ell _{q}$ and $Y\supset \ell _{p}$, with $1\leq
p<\infty $ and $1/p+1/q=1$, then $\mathcal{L}\left( X,Y\right) $ has a
quotient isomorphic to $\ell _{2}$.
\end{enumerate}
\end{theorem}

\proof%
Let us prove each of these statements.

\begin{enumerate}
\item First observe that $X^{\ast }$ is complemented in $\mathcal{L}\left(
X,Y\right) $. Indeed, choose $y_{0}\in Y$ with $\left\| y_{0}\right\| =1$
and apply the Hahn-Banach theorem to get $y_{0}^{\ast }\in Y^{\ast }$ such
that $\left\| y_{0}^{\ast }\right\| =1$ and $y_{0}^{\ast }y_{0}=1$. The map $%
\varphi :X^{\ast }\rightarrow \mathcal{L}\left( X,Y\right) $ defined by $%
\left( \varphi x^{\ast }\right) \left( x\right) =x^{\ast }x\cdot y_{0}$ for
every $x\in X$ is a linear isometry into $\mathcal{L}\left( X,Y\right) $
(see \cite[39.1.(2')]{Ko}), and the operator $P:\mathcal{L}\left( X,Y\right)
\rightarrow \mathcal{L}\left( X,Y\right) $ given by $PT=\varphi \left(
y_{0}^{\ast }\circ T\right) $ is a norm one linear projection operator from $%
\mathcal{L}\left( X,Y\right) $ onto $\mathrm{Im\,}\varphi $. Hence $X^{\ast
} $ is linearly isometric to a norm one complemented linear subspace of $%
\mathcal{L}\left( X,Y\right) $. Since $X^{\ast }$ is a dual Banach space, it
has a separable quotient by Theorem \ref{Dodos}. Hence the operator space $%
\mathcal{L}\left( X,Y\right) $ has a separable quotient.

\item Since $X^{\ast }\otimes _{\varepsilon }Y$ is isometrically embedded in
$\mathcal{L}\left( X,Y\right) $ and both $X^{\ast }$ and $Y$ are
isometrically embedded in $X^{\ast }\otimes _{\varepsilon }Y$, if either $%
X^{\ast }\supset c_{0}$ or $Y\supset c_{0}$, then $\mathcal{L}\left(
X,Y\right) \supset c_{0}$. In this case, according to \cite[Corollary 1]{Fe1}%
, $\mathcal{L}\left( X,Y\right) $ contains an isomorphic copy of $\ell
_{\infty }$. This ensures that $\mathcal{L}\left( X,Y\right) $ has a
separable quotient isomorphic to $\ell _{2}$.

\item If $\left\{ e_{n}:n\in \mathbb{N}\right\} $ is the unit vector basis
of $\ell _{p}$, define $T_{n}:\ell _{p}\rightarrow \ell _{p}$ by $T_{n}\xi
=\xi _{n}e_{n}$ for each $n\in \mathbb{N}$. Since%
\begin{equation*}
\left\| \sum_{i=1}^{n}a_{i}T_{i}\right\| =\sup_{\left\| \xi \right\|
_{p}\leq 1}\left( \sum_{i=1}^{n}\left| a_{i}\xi _{i}\right| ^{p}\right)
^{1/p}\leq \sup_{1\leq i\leq n}\left| a_{i}\right|
\end{equation*}%
for any scalars $a_{1},\ldots ,a_{n}$, we can see that $\left\{ T_{n}:n\in
\mathbb{N}\right\} $ is a basic sequence in $\mathcal{K}\left( \ell
_{p},\ell _{p}\right) $ equivalent to the unit vector basis of $c_{0}$.
Since it holds in general that $E^{\ast }\widehat{\otimes }_{\varepsilon }F=%
\mathcal{K}\left( E,F\right) $ for Banach spaces $E$ and $F$ whenever $%
E^{\ast }$ has the approximation property, if $1/p+1/q=1$ one has that
\begin{equation*}
\ell _{q}\,\widehat{\otimes }_{\varepsilon }\ell _{p}=\ell _{p}^{\ast }\,%
\widehat{\otimes }_{\varepsilon }\ell _{p}=\mathcal{K}\left( \ell _{p},\ell
_{p}\right)
\end{equation*}%
isometrically. So we have $\ell _{q}\,\widehat{\otimes }_{\varepsilon }\ell
_{p}\supset c_{0}$. As in addition $\ell _{q}\,\widehat{\otimes }%
_{\varepsilon }\ell _{p}$ is isometrically embedded in $X^{\ast }\,\widehat{%
\otimes }_{\varepsilon }Y$, which in turn is also isometrically embedded in $%
\mathcal{L}\left( X,Y\right) $, we conclude that $\mathcal{L}\left(
X,Y\right) \supset c_{0}$. So, we use again \cite[Corollary 1]{Fe1} to
conclude that $\mathcal{L}\left( X,Y\right) \supset \ell _{\infty }$. Thus $%
\mathcal{L}\left( X,Y\right) $ has a quotient isomorphic to $\ell _{2}$.
\endproof%
\end{enumerate}

The Banach space $\mathcal{L}\left( X,Y\right) $ need not contain a copy of $%
\ell _{\infty }$ in order to have a separable quotient, as the following
example shows.

Let $1<p,q<\infty $ with conjugated indices $p^{\prime },q^{\prime }$,
i.\thinspace e., $1/p+1/p^{\prime }=1/q+1/q^{\prime }=1$.

\begin{example}
If $p>q^{\prime }$ then $\mathcal{L}\left( \ell _{p},\ell _{q^{\prime
}}\right) $ does not contain an isomorphic copy of $c_{0}$.
\end{example}

\proof%
Since it holds in general that $\mathcal{L}\left( X,Y^{\ast }\right) =(X\,%
\widehat{\otimes }_{\pi }Y)^{\ast }$ isometrically for arbitrary Banach
spaces $X$ and $Y$ (see for instance \cite[Section 2.2]{Ry}), the fact that $%
\ell _{q^{\prime }}^{\ast }=\ell _{q}$ assures that $\mathcal{L}\left( \ell
_{p},\ell _{q^{\prime }}\right) =(\ell _{p}\,\widehat{\otimes }_{\pi }\ell
_{q})^{\ast }$ isometrically. Now let us assume by contradiction that $%
\mathcal{L}\left( \ell _{p},\ell _{q^{\prime }}\right) \supset c_{0}$, which
implies that $\ell _{p}\,\widehat{\otimes }_{\pi }\ell _{q}$ contains a
complemented copy of $\ell _{1}$ (see \cite[Chapter 5, Theorem 10]{Di}).
Since $p>q^{\prime }$, according to \cite[Corollary 4.24]{Ry} or %
\cite[Chapter 8, Corollary 5]{DU}, the space $\ell _{p}\,\widehat{\otimes }%
_{\pi }\ell _{q}$ is reflexive, which contradicts the fact that it has a
quotient isomorphic to the non reflexive space $\ell _{1}$. So we must
conclude that $\mathcal{L}\left( \ell _{p},\ell _{q^{\prime }}\right)
\not\supset c_{0}$.

On the other hand, since $\mathcal{L}\left( \ell _{p},\ell _{q^{\prime
}}\right) $ is a dual Banach space, Theorem \ref{Dodos} shows that $\mathcal{%
L}\left( \ell _{p},\ell _{q^{\prime }}\right) $ has a separable quotient.
Alternatively, we can also apply the first statement of Theorem \ref{Viola}.
\endproof%

\begin{proposition}
If $X^{\ast }$ has the approximation property, the Banach space $\mathcal{N}%
\left( X,Y\right) $ of nuclear operators has a separable quotient.
\end{proposition}

\proof%
Since $X^{\ast }$ enjoys the approximation property, it follows that $%
\mathcal{N}\left( X,Y\right) =X^{\ast }\widehat{\otimes }_{\pi }Y$
isometrically. Hence $X^{\ast }$ is linearly isometric to a complemented
subspace of $\mathcal{N}\left( X,Y\right) $. Since $X^{\ast }$, as a dual
Banach space, has a separable quotient, the transitivity of the quotient map
yields that $\mathcal{N}\left( X,Y\right) $ has a separable quotient.
\endproof%

\begin{theorem}
The following statements hold.

\begin{enumerate}
\item If $X\supset c_{0}$ and $Y\supset c_{0}$, then $\mathcal{L}_{w^{\ast
}}\left( X^{\ast },Y\right) $ has a quotient isomorphic to $\ell _{2}$.

\item If $X$ has a separable quotient isomorphic to $\ell _{1}$, then $%
\mathcal{L}_{w^{\ast }}\left( X^{\ast },Y\right) $ enjoys the same property.

\item If $\left( \Omega ,\Sigma ,\mu \right) $ is an arbitrary measure space
and $Y\neq \left\{ \mathbf{0}\right\} $, then $\mathcal{L}_{w^{\ast }}\left(
L_{\infty }\left( \mu \right) ,Y\right) $ has a quotient isomorphic to $\ell
_{1}$.

\item If $X^{\ast }\supset c_{0}$ or $Y\supset \ell _{\infty }$, then $%
\mathcal{K}\left( X,Y\right) $ has a quotient isomorphic to $\ell _{2}$.

\item If either $X^{\ast }\supset c_{0}$ or $Y\supset c_{0}$, then $\mathcal{%
K}\left( X,Y\right) $ contains a complemented copy of $c_{0}$.

\item If $X\supset \ell _{\infty }$ or $Y\supset \ell _{\infty }$, then $%
\mathcal{K}_{w^{\ast }}\left( X^{\ast },Y\right) $ has a quotient isomorphic
to $\ell _{2}$.

\item The space $\mathcal{W}\left( X,Y\right) $ always has a separable
quotient.

\item If $X\supset c_{0}$ and $Y\supset c_{0}$, then $\mathcal{W}\left(
X,Y\right) $ contains a complemented copy of $c_{0}$.
\end{enumerate}
\end{theorem}

\proof%
In many cases it suffices to show that the corresponding Banach space
contains an isomorphic copy of $\ell _{\infty }$.

\begin{enumerate}
\item By \cite[Theorem 1.5]{Fe3} if $X\supset c_{0}$ and $Y\supset c_{0}$
then $\mathcal{L}_{w^{\ast }}\left( X^{\ast },Y\right) \supset \ell _{\infty
}$.

\item Choose $y_{0}\in Y$ with $\left\| y_{0}\right\| =1$ and select $%
y_{0}^{\ast }\in Y^{\ast }$ such that $\left\| y_{0}^{\ast }\right\| =1$ and
$y_{0}^{\ast }y_{0}=1$. The map $\psi :X\rightarrow \mathcal{L}_{w^{\ast
}}\left( X^{\ast },Y\right) $ given by $\psi \left( x\right) \left( x^{\ast
}\right) =x^{\ast }x\cdot y_{0}$, for $x^{\ast }\in X^{\ast }$, is
well-defined and if $x_{d}^{\ast }\rightarrow x^{\ast }$ under the weak*-
topology of $X^{\ast }$ then $\psi \left( x\right) \left( x_{d}^{\ast
}\right) \rightarrow \psi \left( x\right) \left( x^{\ast }\right) $ weakly
in $Y$, so that $\psi $ embeds $X$ isometrically in $\mathcal{L}_{w^{\ast
}}\left( X^{\ast },Y\right) $. On the other hand, the operator $Q:\mathcal{L}%
_{w^{\ast }}\left( X^{\ast },Y\right) \rightarrow \mathcal{L}_{w^{\ast
}}\left( X^{\ast },Y\right) $ given by $QT=\psi \left( y_{0}^{\ast }\circ
T\right) $, which is also well-defined since $y_{0}^{\ast }\circ T\in X$
whenever $T$ is weak*-weakly continuous, is a bounded linear projection
operator from $\mathcal{L}_{w^{\ast }}\left( X^{\ast },Y\right) $ onto $%
\text{Im}\psi $. Since we are assuming that $\ell _{1}$ is a quotient of $X$%
, it follows that $\ell _{1}$ is also isomorphic to a quotient of $\mathcal{L%
}_{w^{\ast }}\left( X^{\ast },Y\right) $.

\item This statement is consequence of the previous one, since $\ell _{1}$
embeds complementably in $L_{1}\left( \mu \right) $.

\item According to \cite{Ka}, if $X^{\ast }\supset c_{0}$ or $Y\supset \ell
_{\infty }$, then $\mathcal{K}\left( X,Y\right) \supset \ell _{\infty }$.

\item This property has been shown in \cite[Corollary 1]{Rya}.

\item The map $\psi :X\rightarrow \mathcal{L}_{w^{\ast }}\left( X^{\ast
},Y\right) $ defined above by $\psi \left( x\right) \left( x^{\ast }\right)
=x^{\ast }x\cdot y_{0}$, for every $x^{\ast }\in X^{\ast }$, yields a
finite-rank (hence compact) operator $\psi \left( x\right) $, so that $\text{%
Im}\psi \subset \mathcal{K}_{w^{\ast }}\left( X^{\ast },Y\right) $. On the
other hand, if $x_{0}\in X$ with $\left\| x_{0}\right\| =1$ and $x_{0}^{\ast
}\in X^{\ast }$ verifies that $\left\| x_{0}^{\ast }\right\| =1$ and $%
x_{0}^{\ast }x_{0}=1$, the map $\phi :Y\rightarrow \mathcal{K}_{w^{\ast
}}\left( X^{\ast },Y\right) $ given by $\phi \left( y\right) \left( x\right)
=x_{0}^{\ast }x\cdot y$, for every $x\in X$, is a linear isometry from $Y$
into $\mathcal{K}_{w^{\ast }}\left( X^{\ast },Y\right) $. Hence $X$ and $Y$
are isometrically embedded in $\mathcal{K}_{w^{\ast }}\left( X^{\ast
},Y\right) $.

\item Just note that $\mathcal{W}\left( X,Y\right) =\mathcal{L}_{w^{\ast
}}\left( X^{\ast \ast },Y\right) $ isometrically. Since $X^{\ast }$ is
complementably embedded in $\mathcal{L}_{w^{\ast }}\left( X^{\ast \ast
},Y\right) $, the conclusion follows from Theorem \ref{Dodos}.

\item According to \cite[Theorem 2.5]{Fe4}, under those conditions the space
$\mathcal{W}\left( X,Y\right) $ contains a complemented copy of $c_{0}$.
\endproof%
\end{enumerate}

\begin{remark}
\emph{If neither }$X$\emph{\ nor }$Y$\emph{\ contains a copy of }$c_{0}$%
\emph{, then }$\mathcal{L}_{w^{\ast }}\left( X^{\ast },Y\right) $\emph{\
cannot contain a complemented copy of }$c_{0}$\emph{\ as observed in \cite%
{Em2}.}
\end{remark}

The following result sharpens the first statement of Theorem \ref{Vera}.

\begin{corollary}
If $\left( \Omega ,\Sigma ,\mu \right) $ is a finite measure space, $%
\widehat{P_{1}\left( \mu ,X\right) }$ has a quotient isomorphic to $\ell
_{1} $.
\end{corollary}

\proof%
This follows from the second statement of the previous theorem together with
the fact that $\widehat{P_{1}\left( \mu ,X\right) }=\mathcal{L}_{w^{\ast
}}\left( L_{\infty }\left( \mu \right) ,X\right) $ (see \cite[Chapter 8,
Theorem 5]{DU}).

\begin{remark}
The space $P_{1}\left( \mu ,X\right) $ need no contain a copy of $\ell
_{\infty }$. \emph{It can be easily shown that }$P_{1}\left( \mu ,X\right) $%
\emph{\ embeds isometrically in the space }$\mathcal{K}_{w^{\ast }}\left(
ca\left( \Sigma \right) ^{\ast },X\right) $\emph{, where }$ca\left( \Sigma
\right) $\emph{\ denotes the Banach space of scalarly-valued countably
additive measures equipped with the variation norm. Since }$ca\left( \Sigma
\right) \not\supset \ell _{\infty }$\emph{, it follows from \cite[Theorem]%
{Dr2} that }$P_{1}\left( \mu ,X\right) \supset \ell _{\infty }$\emph{\ if
and only if }$X\supset \ell _{\infty }$\emph{.}
\end{remark}

\section{Separable quotients in spaces of vector measures}

In this section we denote by $ba\left( \Sigma ,X\right) $ the Banach space
of all $X$-valued bounded finitely additive measures $F:$ $\Sigma
\rightarrow X$ provided with the semivariation norm $\left\| F\right\| $.
The closed linear subspace of $ba\left( \Sigma ,X\right) $ consisting of
those countably additive measures is represented by $ca\left( \Sigma
,X\right) $, while $cca\left( \Sigma ,X\right) $ stands for the (closed)
linear subspace of $ca\left( \Sigma ,X\right) $ of all measures with
relatively compact range. It can be easily shown that $ca\left( \Sigma
,X\right) =\mathcal{L}_{w^{\ast }}\left( ca\left( \Sigma \right) ^{\ast
},X\right) $ isometrically. We also design by $bvca\left( \Sigma ,X\right) $
the Banach space of all $X$-valued countably additive measures $F:$ $\Sigma
\rightarrow X$ of bounded variation equipped with the variation norm $\left|
F\right| $. Finally, following \cite[page 107]{Ry}, we denote by $\mathcal{M}%
_{1}\left( \Sigma ,X\right) $ the closed linear subspace of $bvca\left(
\Sigma ,X\right) $ consisting of all those $F\in bvca\left( \Sigma ,X\right)
$ that have the so-called \textit{Radon-Nikod\'{y}m property}, i.\thinspace
e., such that for each $\lambda \in ca^{+}\left( \Sigma \right) $ with $F\ll
\lambda $ there exists $f\in L_{1}\left( \lambda ,X\right) $ with $F\left(
E\right) =\int_{E}f\,d\lambda $ for every $E\in \Sigma $. For this section,
our main references are \cite{DU} and \cite{Ry}.

\begin{theorem}
The following statements hold. In the first case $X$ need not be
infinite-dimensional.

\begin{enumerate}
\item If $X\neq \left\{ \mathbf{0}\right\} $, then $ba\left( \Sigma
,X\right) $ always has a separable quotient.

\item If $X\supset c_{0}$, then $ba\left( \Sigma ,X\right) $ has a quotient
isomorphic to $\ell _{2}$.

\item If $X\supset c_{0}$ but $X\not\supset \ell _{\infty }$, then $ba\left(
\Sigma ,X\right) $ has a complemented copy of $c_{0}$.

\item If $\Sigma $ admits no atomless probability measure, then $ca\left(
\Sigma ,X\right) $ has a quotient isomorphic to $\ell _{1}$.

\item If $X\supset c_{0}$ and $\Sigma $ admits a nonzero atomless $\lambda
\in ca^{+}\left( \Sigma \right) $, then $ca\left( \Sigma ,X\right) $ has a
quotient isomorphic to $\ell _{2}$.

\item If there exists some $F\in cca\left( \Sigma ,X\right) $ of unbounded
variation, then $cca\left( \Sigma ,X\right) $ has a separable quotient.

\item If $X\supset c_{0}$, then $cca\left( \Sigma ,X\right) $ contains a
complemented copy of $c_{0}$.

\item If $X\supset \ell _{1}$, then $\mathcal{M}_{1}\left( \Sigma ,X\right) $
has a quotient isomorphic to $\ell _{1}$.
\end{enumerate}
\end{theorem}

\proof%
In cases 2 and 3 it suffices to show that the corresponding Banach space
contains an isomorphic copy of $\ell _{\infty }$.

\begin{enumerate}
\item This happens because $ba\left( \Sigma ,X\right) =\mathcal{L}\left(
\ell _{\infty }\left( \Sigma \right) ,X\right) $ isometrically. Since $\ell
_{\infty }\left( \Sigma \right) $ is infinite-dimensional by virtue of the
non triviality of the $\sigma $-algebra $\Sigma $, the statement follows
from the first statement of Theorem \ref{Viola}.

\item By point 2 of Theorem \ref{Viola}, if $X\supset c_{0}$ then $\mathcal{L%
}\left( \ell _{\infty }\left( \Sigma \right) ,X\right) $ has a quotient
isomorphic to $\ell _{2}$. The statement follows from the fact that $%
ba\left( \Sigma ,X\right) =\mathcal{L}\left( \ell _{\infty }\left( \Sigma
\right) ,X\right) $.

\item $ba\left( \Sigma ,X\right) $ has a complemented copy of $c_{0}$ by
virtue of \cite[Corollary 3.2]{Fe4}.

\item If the non trivial $\sigma $-algebra $\Sigma $ admits no atomless
probability measure, it can be shown that $ca\left( \Sigma ,X\right) $ is
linearly isometric to $\ell _{1}\left( \Gamma ,X\right) $ for some infinite
set $\Gamma $. Since $\ell _{1}\left( \Gamma ,X\right) =L_{1}\left( \mu
,X\right) $, where $\mu $ is the counting measure on $2^{\Gamma }$, the
conclusion follows from the third statement of Theorem \ref{Cordelia}.

\item Since $ca\left( \Sigma \right) \,\widehat{\otimes }_{\varepsilon
}X=cca\left( \Sigma ,X\right) $ isometrically, if $X\supset c_{0}$ then $%
cca\left( \Sigma ,X\right) \supset c_{0}$ and hence $ca\left( \Sigma
,X\right) \supset c_{0}$. If $\Sigma $ admits a nonzero atomless $\lambda
\in ca^{+}\left( \Sigma \right) $, then one has $ca\left( \Sigma ,X\right)
\supset \ell _{\infty }$ by virtue of \cite[Theorem 1]{Dr}.

\item Observe that $cca\left( \Sigma ,X\right) =ca\left( \Sigma \right) \,%
\widehat{\otimes }_{\varepsilon }X$ and $\mathcal{M}_{1}\left( \Sigma
,X\right) =ca\left( \Sigma \right) \,\widehat{\otimes }_{\pi }X$
isometrically (see \cite[Theorem 5.22]{Ry}) but, at the same time, from the
algebraic point of view, $\mathcal{M}_{1}\left( \Sigma ,X\right) $ is a
linear subspace of $cca\left( \Sigma ,X\right) $ since every Bochner
indefinite integral has a relatively compact range, \cite[Chapter II,
Corollary 9 (c)]{DU}. If there exists some $F\in cca\left( \Sigma ,X\right) $
of unbounded variation, then $\mathcal{M}_{1}\left( \Sigma ,X\right) \neq
cca\left( \Sigma ,X\right) $, so the statement follows from Theorem \ref%
{Miranda}.

\item Since $X\supset c_{0}$ and $ca\left( \Sigma \right) $ is
infinite-dimensional, then $X$\thinspace $\widehat{\otimes }_{\varepsilon
}ca\left( \Sigma \right) $ contains a complemented copy of $c_{0}$ by %
\cite[Theorem 2.3]{Fre}.

\item Since $\mathcal{M}_{1}\left( \Sigma ,X\right) =ca\left( \Sigma \right)
\,\widehat{\otimes }_{\pi }X$ and $ca\left( \Sigma \right) \supset \ell _{1}$%
, if $X\supset \ell _{1}$ then $\mathcal{M}_{1}\left( \Sigma ,X\right) $ has
a quotient isomorphic to $\ell _{1}$ as follows from Theorem \ref{Bombal}.
\endproof%
\end{enumerate}

\begin{remark}
\emph{If }$\omega \in \Omega $\emph{\ and }$E\left( \Sigma ,X\right) $\emph{%
\ is either }$ba\left( \Sigma ,X\right) $\emph{, }$ca\left( \Sigma ,X\right)
$\emph{\ or }$bvca\left( \Sigma ,X\right) $\emph{, the map }$P_{\omega
}:E\left( \Sigma ,X\right) \rightarrow E\left( \Sigma ,X\right) $\emph{\
defined by }$P_{\omega }\left( F\right) =F\left( \Omega \right) \delta
_{\omega }$\emph{\ is a bounded linear projection operator onto the copy }$%
\left\{ x\,\delta _{\omega }:x\in X\right\} $\emph{\ of }$X$\emph{\ in }$%
E\left( \Sigma ,X\right) $\emph{. Hence, if }$X$\emph{\ has a separable
quotient isomorphic to }$Z$\emph{, then }$E\left( \Sigma ,X\right) $\emph{\
also has a separable quotient isomorphic to }$Z$\emph{.}
\end{remark}

\begin{remark}
$cca\left( \Sigma ,X\right) $ may not have a copy of $\ell _{\infty }$.
\emph{Since }$cca\left( \Sigma ,X\right) =\mathcal{K}_{w^{\ast }}\left(
ca\left( \Sigma \right) ^{\ast },X\right) $\emph{, according to %
\cite[Theorem or Corollary 4]{Dr2}, }$cca\left( \Sigma ,X\right) \supset
\ell _{\infty }$\emph{\ if and only if }$X\supset \ell _{\infty }$\emph{.}
\end{remark}

\begin{remark}
\emph{Concerning the space }$\mathcal{M}_{1}\left( \Sigma ,X\right) $\emph{,
it is worthwhile to mention that it follows from \cite[Theorem]{Fe} that }$%
\mathcal{M}_{1}\left( \Sigma ,X\right) \supset \ell _{\infty }$ \emph{if and
only if }$X\supset \ell _{\infty }$\emph{.}
\end{remark}

\end{document}